\documentclass[12pt,a4paper]{article}
\usepackage{graphicx}
\usepackage[T1]{fontenc}
\usepackage[utf8]{inputenc}
\usepackage{color}
\definecolor{darkgreen}{rgb}{0,0.51,0.11}
\usepackage{amsmath,amsthm,amsfonts,amssymb,amscd}

\numberwithin{equation}{section}

\def\RR{{\mathcal{R}}}
\def\QQ{{\mathcal{Q}}}
\def\ZZ{{\mathbb{Z}}}
\def\NN{{\mathbb{N}}}
\def\x{\tilde{x}}
\def\y{\tilde{y}}
\def\r{\tilde{r}}

\def\z{\bar{z}}
\def\p{\bar{p}}
\def\q{\bar{q}}
\def\fmx{\tilde{f}^{m}(\tilde{x})}
\def\fnx{\tilde{f}^{n}(\tilde{x})}
\def\fmy{\tilde{f}^{m}(\tilde{y})}
\def\fny{\tilde{f}^{m}(\tilde{y})}
\def\fnx*{\tilde{f}^{-n}(\tilde{x})}
\def\fmx*{\tilde{f}^{-m}(\tilde{x})}
\def\fmy*{\tilde{f}^{-m}(\tilde{y})}
\def\fny*{\tilde{f}^{-m}(\tilde{y})}

\newtheorem{theorem}{Theorem}[]

\newtheorem{corollary}{Corollary}[theorem]
\newtheorem{lemma}{Lemma}[section]
\newtheorem{proposition}{Proposition}[section]
\newtheorem{definition}{Definition}[section]
\newtheorem{remark}{Remark}[section]
\newtheorem{question}{Question}

\begin{document}











\title{On the Uniqueness of SRB Measures for Endomorphisms}
\author{P. Mehdipour \footnote{Post-doc at Federal University of Itajuba}}

\maketitle
\begin{abstract}
In this paper we improve the results of \cite{MT} and show that a weak hyperbolic transitivity implies the uniqueness of hyperbolic SRB measures. As an important corollary, it arises the ergodicity of the system in a conservative setting. It also arises the condition which implies the stable ergodicity as well as the statistical stability for a general $C^2$-regular map. 
\end{abstract}

\renewcommand{\thefootnote}{\fnsymbol{footnote}}
\footnote[0]{2010 Mathematics Subject classification 58K05, 53A04, 57R45.}

\section{\textbf{Introduction}}\label{intro}
The contrast of topological and measure theoretical properties is an interesting subject which frequently appears in the study of dynamics. 

In a valuable simple construction, I. Kan \cite{5} gave an example of a local diffeomorphism $f$ defined on the cylinder $\mathbb{S}^1 \times [0,1]$ such that $f$ can be topologically transitive and admits two SRB measures with intermingled basins. The richness of intermingled basins property and the non-uniqueness of the SRB measures together with topological transitivity is a nice phenomenon. We would like to mention that by a "SRB measure" we mean measures that their unstable conditional parts are absolutely continuous with respect to Lebesgue (see \cite{MT} for a precise definition). Once assuming dynamics bearing hyperbolic measures (ergodic measures bearing non-zero Lyapunov exponents), then such SRB measures are also "phyisical measures" (a measure with positive Lebesgue basin of attraction).

Along the results obtained in \cite{13} and \cite{MT} 
respectively about the uniqueness of SRB measures for transitive diffeomorphisms of surfaces and the maximum number of SRB measures for endomorphisms, a forward step seems to be a condition which implies the uniqueness of SRB measures for endomorphisms of the same type and surely such a condition should not outcrop in examples of type I.Kan \cite{5}.


Let $f$ be a $C^2$ regular map (local diffeomorphism) on a compact closed Riemannian manifold and $\mu$ any hyperbolic invariant measure (for a definition see \S\ref{preliminaries} and for simplicity we may call such pair $(f,\mu)$ a "measure dynamics"). We show that either a weak hyperbolic transitivity or any weak hyperbolic pre-transitivity is enough to achieve the SRB uniqueness property for endomorphisms. By a weak hyperbolic transitivity we mean the existence of a non-uniformely hyperbolic dense total orbit and by pre-transitivity, we mean the density of the total orbit of a periodic point which may commonly occurs for endomorphisms. By a total orbit $O_T(.)$ we simply mean the set of all images and pre-images of a point in the phase space.

There are results in which one can find some uniqueness results for Axiom A attractors
of a $C^2$-endomorphism and Axiom A endomorphisms (see \cite{1}, \cite{2}). What we are interested here is mostly related to the uniqueness of SRB measures for the case of non-uniformly hyperbolic endomorphisms as well as attractors of Milnor types \cite{M}. 

Despite the fact that examples of type I. Kan are topologically transitive \cite{BDV}, but we do not know any such example with topologically mixing propert. In this way it arises a conjecture that having the topologically mixing property, may imply the uniqueness of SRB measures. In this paper our intension is to show that by a weak hyperbolic transitivity (which in case can be a weaker condition than the mixing property) one may obtain the uniqueness of SRB measures. In other words we can say that the non-robustness property of Kan-type examples on boundary-less manifolds have relation to their non-robustness of weak hyperbolic transitivity. In this direction, the recent work of C. Bonatti and R. Potrie \cite{BP2} is valuable since in a partially hyperbolic setting with mostly contracting dominated splitting, they construct examples of 3-dimensional torus with infinitely many intermingled basins and mixing property. Their construction does not work in a non-uniformly hyperbolic setting of our type. 

Considering the conditions we give here, to obtain the uniqueness property of SRB measures, the following question seems quite natural:

\begin{question}
Is there any topologically mixing measure dynamics with a unique hyperbolic SRB measure and without a weak hyperbolic transitivity?
\end{question}

Some other recent results Due to Z. Lian, P. Liu and K. lu \cite{LLL}. They prove the uniqueness of SRB measures for a class of uniformly hyperbolic and mixing partially hyperbolic diffeomorphisms. 

The idea of using hyperbolic periodic points to analyze the number of SRB measures appears in \cite{13} and \cite{VY1}. In the context of partially hyperbolic diffeomorphisms with mostly contracting center, Viana and Yang \cite{VY} exhibited skeleton (defined by them with some similar properties) determining the number of basins of physical measures. They also concluded continuity results about the number of physical measures. Once working at the world of diffeomorphisms, the notion of a hyperbolic periodic point is well defined. But in case of a hyperbolic local diffeomorphism this notion plays a more concrete role by looking at the subset of the pre-images of a hyperbolic point $p$. They are copies of the original periodic point which in case can appear with same properties. As a simple example take a hyperbolic periodic point and look at its Lyapunov exponent in comparison with its pre-periodics one. Their forward Lyapunov exponents will coincide and if we use this, to construct an stable sub-bundle along the total orbit of $p$. Once defining a homoclinic relation between different periodic points by a transversality condition along their stable and unstable sets, then this same stable-sub bundle property of a total orbit of $p$ may help us to make relation between periodic points.

The first version of following theorem was proved under the condition of a general pre-transitivity of a point, but due to the Dichotomy of Transitivity \cite{LPV}, the pre-transitivity condition may reduces to the transitivity ( i.e. the existence of a forward dense orbit).


\begin{theorem} \label{main1}
Let $f: M \rightarrow M$ be a $C^2-$regular map of a closed $2-$dimensional manifold. Then a weak hyperbolic transitivity implies uniqueness of the (hyperbolic) SRB measure. (It works for $dim M=n$, up to a fixed stable $k-$index)
\end{theorem}



The following theorem may be useful in constructing a unique SRB measure for endomorphisms.

\begin{theorem}\label{2}
Let $(f,\mu)$ be a measure dynamics with fixed $k-$index (stable index). Then hyperbolic pre-transitivity implies the uniqueness of SRB measures. 	
\end{theorem}

 Following Theorem and Corollary emerge around the ergodicity and ergodic stability.
\begin{theorem} \label{main2}
	Let $f: M \rightarrow M$ be a $C^2-$ conservative regular map on a closed $n-$dimensional manifold (fixed $k-$index). Then hyperbolic transitivity, implies the ergodicity.
\end{theorem}

\begin{corollary} \label{main3}
	Let $f: M \rightarrow M$ be a $C^r-$ regular map of a closed $n-$dimensional manifold. Then a hyperbolic robust  transitivity implies the statistic stability and in case of conservative maps, the ergodic stability of the map.
\end{corollary}

The paper is organized as follows. Some preliminaries on endomorphisms are presented in Section 2 while in Section 3 some known results are emphasized in the context of results obtained in \cite{MT}. In Section 4 the main results (Theorems 1, 2 and 3, Corollary 3.1) are proved. In Section 5 we give some examples.  




\section{Some Preliminaries on Endomorphisms} \label{preliminaries}
Let $M$ be a closed Riemannian manifold. By a $C^2-$endomorphisms $f:M\rightarrow M$ we mean a $C^{2}-$local diffeomorphism or a $C^2-$regular map ($|d_xf|\neq 0$, where $|.|$ is the determinant at point $x$). 

\textbf{\textit{Inverse Limit Space.}}
Let $f:M\rightarrow M$ be a $C^2-$regular map and let $\mathcal{M}_{f}(M)$ denotes the set of all $f-$invariant Borel probability measures.
For this $f$, consider the compact metric space
 $$M^{f}:=\{\tilde{x}=(x_{n})\in\prod_{-\infty}^{\infty}M : f(x_{n})=x_{n+1} \quad \text{for all} \quad n\in \mathbb{Z}\},$$
 equipped with the distance $\tilde{d},$ given by
 $$\tilde{d}(\tilde{x},\tilde{y}):=\sum_{n=-\infty}^{\infty}2^{-|n|}d(x_{n}, y_{n}),$$
 where $\tilde{x}=(x_{n}),\tilde{y}=(y_{n})\in M^{f}$ and $d$ is the distance on $M$ induced by the Riemannian metric. Such metric space is called the "Inverse Limit Space" of $f$. Let $\pi$ be the natural projection from $M^{f}$ to $M$ i.e, $\pi((x_{n}))=x_{0}, \forall \tilde{x}\in M^{f}$ and $\tilde{f}:M^{f}\rightarrow M^{f}$ be the shift homeomorphism. It is clear that  $\pi\circ \tilde{f}=f\circ \pi.$ 
 
 The $(M,f,\mu)$ is a probability space where $\mu$ is an $f-$invariant probability measure defined on the Borel $\sigma-$algebra induced by the metric property of $M$ and $f$ a $C^2-$regular map that as mentioned before, for simplicity we call it a "measure dynamics", and denote it by $(f,\mu)$ in many parts of the text.
 The invariant measure $\mu$ is called an ergodic measure, if any $f-$invariant subset of the $\sigma-$algebra, have $\mu-$measure 0 or 1. It happens that a dynamical system is ergodic with respect to some of its invariant measures. In a volume preserving (or conservative) setting, what plays an important feature, is the ergodicity of the "Lebesgue" measure. Hence by an "ergodic dynamical system", we mean an ergodic dynamics $(f,m)$ where $m$ is the Lebesgue measure, induced from Riemannian structure of the manifold.
 
 The following proposition of \cite{MT} is important.
 
 \begin{proposition}\label{1-1}
 	Let $f$ be a continuous map on $M$. For any $f-$invariant Borel probability measure $\mu$ on $M$, there exists a unique $\tilde{f}-$invariant Borel probability measure $\tilde{\mu}$ on $M^{f}$ such that $\pi_{*} \tilde{\mu}=\mu$. Moreover, $\mu$ is ergodic if and only if $\tilde{\mu}$ is ergodic.
 \end{proposition}

\textbf{\textit{Hyperbolicity.}} Let $f$ be a $C^{2}-$endomorphism and $\Lambda$ be an $f-$invariant closed subset of $M$. One defines 
$$\Lambda^{f}:=\{\tilde{x}=(x_{n})\in M^{f} :  x_{n}\in\Lambda, \,\,for\, all\,\, n\in \mathbb{Z}\}.$$

Then $\Lambda$ is called a \textit{Hyperbolic Set} if there exists real constants $C>0$ and $0<\mu<1$ such that $\forall\x\in\Lambda^{f}$ and for every integer $n. m\in \ZZ^{+}$ we have:
\begin{itemize}
	\item $T_{x_{n}}M= E^{s}(\tilde{x},n)\oplus E^{u}(\tilde{x},n)$,
	\item $Df(E^{s}(\tilde{x},n))= E^{s}(\tilde{f}(\tilde{x}),n)=E^{s}(\tilde{x},n+1),$\\
	$\Vert Df^{m}_{x_{n}}(v)\Vert\leq C\mu^{m}\Vert v\Vert,$ for $\,v\in E^{s}_{x_{n}}$,
	\item $Df(E^{u}(\tilde{x},n))= E^{u}(\tilde{f}(\tilde{x}),n)=E^{u}(\tilde(x),n+1),$\\
	$ \Vert Df^{m}_{x_{n}}(v)\Vert\geq [C\mu^{m}]^{-1}\Vert v\Vert,$ for $\,v\in E^{u}_{x_{n}}.$
\end{itemize}

In above definition an example can be $\Lambda=\{p,f(p),...,f^{n}(p)\}$, the real orbit of a periodic point of period $n.$ Then its hyperbolicity means the hyperbolicy of $\Lambda^{f}=\{\bar p\}$ by definition. We define the "\textbf{Total orbit}" of a periodic point as
$$O_T(p)=\{f^{n}(p): n\in \ZZ\}.$$

In a similar way the total orbit of any other point $x$, denoting by $O_T(x)$, can be defined (by $f^{-1}$ we mean the pre-image).
We say that $f$ is an \textit{"Anosov map"} if $f$ is hyperbolic on $M$. 

Now for $x\in M$ and $f:M\rightarrow M$ as above, suppose $v\in T_{x}M$. We put
$$\lambda(x,v)=\overline{\lim}_{n\rightarrow\infty}\frac{1}{n} \log\Vert d_{x}f^{n}(v)\Vert.$$
The number $\lambda(x,v)$ is called the \textit{Lyapunov Exponent for $v$}. For $x\in M$ there are
at most (dim$M$)-numbers $\lambda_{1}(x),$ $\cdots,$ $\lambda_{r(x)}(x)$ with $-\infty<\lambda_{1}(x)<\cdots<\lambda_{r(x)}<\infty.$ Also there is a
filtration of subspaces
$$\{0\}=L_{0}(x)\subsetneq L_{1}(x)\subsetneq\cdots\subsetneq L_{r(x)}(x)=T_{x}M,$$
where $L_{i}(x)=\{v\in T_{x}M: \lambda(x,v)\leq \lambda_{i}(x),\,\textit{for}\, 1\leq i\leq r(x)\}$.

We have $\lambda(x,v)=\lambda_{i}(x)$, holds for $v\in L_{i}(x)\backslash L_{i-1}(x),$ $1\leq i\leq r(x)$ and the numbers $\lambda_{1}(x)$ ,$\cdots,$
$\lambda_{r(x)}(x)$ are called the \textit{Lyapunov exponents at $x$}. We set
$$m_{i}(x)=\dim L_{i}(x) -\dim L_{i-1}(x)$$
for $1\leq i\leq r(x)$ and it is the \textit{multiplicity} of Lyapunov exponent $\lambda_{i}(x)$. Observe that $r(x),$ $\lambda_{i}(x),$ $m_{i}(x)$ are measurable $f$-invariant functions with respect to any $f$-invariant Borel probability measure $\mu$ and once $\mu$ is an ergodic measure, then these functions become constant almost everywhere. In the context of non-uniform hyperbolicity, the Lyapunov exponents support an important pattern.

Before mentioning the definition of a non-uniformely hyperbolic set, let us bring the definition of a hyperbolic measure.

\textbf{\textit{Hyperbolic Measure.}}\label{hyp}
An ergodic $f-$invariant Borel probability measure $\mu$, is called hyperbolic, if:

\begin{itemize}
	\item none of the Lyapunov exponents for $\mu$ are zero;
	\item there exist Lyapunov exponents with different signs.
\end{itemize}


\textbf{\textit{Non-Uniformly Hyperbolic Set (NUH).}} We say that the $f-$invariant subset $\Lambda\subset M,$ is "Non-Uniformly hyperbolic" (NUH) if considering $\Lambda^{f}\subset M^{f}$ that $\pi(\Lambda^{f})=\Lambda$, then there exist \\
(a) numbers $\beta,\,\theta$ such that $0<\beta<1<\theta $;\\
(b) a number $\epsilon>0$ and Borel functions $C,K:\, \Lambda^{f}\rightarrow (0,\infty)$;\\
(c) subspaces $E^s(\tilde{x},n)$ and $E^u(\tilde{x},n)$ for each $\tilde{x}=(x_{n})\in \tilde{\Lambda}$, which satisfy the following conditions $\forall n\in\ZZ$:

\begin{enumerate}
	\item $T_{x_{n}}M=E^s(\tilde{x},n)\oplus E^u(\tilde{x},n)$,\ \ \
	\item
	$\begin{cases}
	d_{x_{n}}f\,E^s(\tilde{x},n)=E^s(\tilde{f}(\tilde{x}),n)=E^{s}(\tilde{x}, n+1);\\
	
	d_{x_{n}}f\,E^u(\tilde{x},n)=E^u(\tilde{f}(\tilde{x}),n)=E^{u}(\tilde{x}, n+1);
	\end{cases}$
	\item the subspace $E^s(\tilde{x},n)$ is stable: \\
	for $v\in E^s(\tilde{x},n)$ and $m>0$,\ \ \ \ \ \
	$\Vert d_{x_{n}}f^{m}(v)\Vert\leq C(\tilde{x}) \beta^{m}e^{\epsilon(m+|n|)}\Vert v \Vert$;
	\item the subspace $E^u(x)$ is unstable: \\
	for $v\in E^u(\tilde{x},n)$ and $m<0$,\ \ \ \ \ \ $\Vert d_{x_{n}}f^m(v)\Vert\leq C(\tilde{x})\theta^{m}e^{\epsilon(|m|+|n|)}\Vert v\Vert$;
	\item $\angle(E^s(\tilde{x},n),E^u(\tilde{x},n))\geq K(\tilde{x})$;
	\item $ C(\tilde{f}^m(\tilde{x}))\leq C(\tilde{x})e^{\epsilon |m|},\,\,\,\,\,\,K(\tilde{f}^m(\tilde{x}))\geq K(\tilde{x})e^{-\epsilon\,|m |}.$
\end{enumerate}

\begin{remark}\label{s-u NUH}
Around the existence of well-defined local stable-unstable manifolds for a NUH set, we refer to Theorems 4.1.1 and 4.1.2 of \cite{7}, which guarantee the existence of local manifolds tangent to invariant sub-bundles. After some suitable adaptations it is possible to obtain such results for endomorphisms specially in case of local diffeomorphisms. We refer to \cite{10} for more related details. 
\end{remark}

By celebrated Oseledet\'s Multiplicative
Ergodic Theorem for $C^1-$differentiable maps \cite{10}, for any $f-$invariant measure $\mu$, there exists a full $\mu-$measure subset of Lyapunov regular points. Due to its adapted results for endomorphisms (see \cite{MT} or \cite{10}), this full measure subset lives on $M^f$ and we denote it by $\tilde{\RR}$. We may denote the adapted version of the theorem for endomorphisms on inverse limit space, by $\widetilde{MET}$.


	\begin{remark}\label{regular NUH set}
		The set of all points from $\widetilde{MET}$ is called \textit{"Lyapunov regular set"}. Let us consider $\mu$ being a hyperbolic measure. Then the set of Lyapunov regular points without zero Lyapunov exponents, contains a non-uniformly hyperbolic set of full $\tilde{\mu}-$measure with 
		$$\beta=\beta^{\mu},\,\,\,\theta=\theta^{\mu},\,\,\,C(\x)=C(\x,\epsilon),\,\,\,K(\x)=K(\x,\epsilon)$$
		
		for $0<\epsilon\leq\epsilon_0.$
		Without loss of generality from now on, let us consider $\tilde{\RR}$ being the full measure, non-uniformly hyperbolic subset of Lyapunov regular points on $M^{f}$, for a measure dynamics $(f,\mu)$, and which, we denote its projection on $M$ by $\RR$ ($\pi(\tilde{\RR})=\RR$).
	\end{remark}
	
	
Let suppose $(f,\mu)$ a measure dynamics with $\mu$ a hyperbolic measure. Let $\beta=\beta^{\mu}$ (resp $\theta=\theta^{\mu}$) be the least in modulus positive (resp. negative) Lyapunov exponents. Suppose that $\mu$ has $k$ negative Lyapunov exponents.\\

	\textbf{\textit{Pesin Blocks.}}\label{1-08} Fix $0 < \epsilon \ll 1.$ Given $l>0$, we define a Pesin block (Regular set) $\tilde{\Delta}_{l}$ as:
	
	$$\tilde{\Delta}_{l}:=\{\x\in \tilde{\RR}: C(\x,\epsilon)\leq l,\,\,\,\,K(\x,\epsilon)\geq\frac{1}{l}\}.$$
	These subsets can be viewed as non-invariant uniformly hyperbolic sets and as a subset of $M^{f}$ consisting of
	$\tilde{x}=(x_{n})\in M^{f}$ for which there exists a sequence of splittings $T_{x_{n}}M=E^{s}(\tilde{x}, n)\oplus E^{u}(\tilde{x}, n)$,$n\in \ZZ$. Some of their properties are:

	
	\begin{itemize}
		\item $\tilde{\Delta}^{l}\subset \tilde{\Delta}^{l+1},$ $\dim E^{s}(\tilde{x}, n)=k$;
		
		\item the subspaces $E^{s}(\tilde{x},n)$ and $E^{u}(\tilde{x},n)$ depend continuously on $\x\in \tilde{\Delta}^{l},$
		\item we consider the closure of Pesin blocks $\tilde{\Delta}^{l}:=\overline{\tilde{\Delta}^{l}}$ and for $\x \in \overline{\tilde{\Delta}^{l}},$\\  $D_{x_{n}}f(E^{s}(\tilde{x},n))= E^{s}(\tilde{x}, n+1),$ $D_{x_{n}}f(E^{u}(\tilde{x}, n))=E^{u}(\tilde{x}, n+1)$;
		\item moreover for $m\geq 0, \, v\in E^{s}(\tilde{x}, n)$ and $w\in E^{u}(\tilde{x}, n);$\\
		
		$\begin{cases}
		\Vert D_{x_{n}}f^{m}(v)\Vert\leq l \beta^m e^{(\epsilon |n|)}\Vert v \Vert,\forall n\in\ZZ,n\geq 1\\
		
		\Vert (D_{x_{n-m}}f^{m}|_{E^{u}(\tilde{x},n-m)})^{-1}(w) \Vert\leq l\theta^m e^{(\epsilon |n-m|)}\Vert w \Vert,\forall n\in\ZZ,n\geq 1;
		\end{cases}$\\
		\item $\angle(E^{s}(\tilde{x}, n),$ $E^{u}(\tilde{x}, n))\geq \frac{1}{l}$.
	\end{itemize}
	Let us denote this new closed, $f$-invariant set by $\tilde{\QQ}:= \bigcup_{l\geq 1}\tilde{\Delta}^{l}.$




\textbf{\textit{Weak Hyperbolicity.}} Let $f:M\to M$ be a $C^2-$ regular map. Then we say that $f$ has a weak hyperbolicity, if there exists some $x\in M$ such that the $f-$invariant total orbit subset $O_{T}(x)$, is non-uniformly hyperbolic due to given definition of NUH.

\begin{remark}[\textit{\textbf{Important}}]\label{important}
Observe that a weak hyperbolic total orbit, $O_{T}(x)$, despite the fact that can contain non-regular orbits, the subset $O_{T}^{f}(x)$ may live in the closure of Pesin blocks as well as it may not. Once they live inside such good blocks, we may get the properties such as continuous variation of local stable-unstable manifolds depending on $\x\in \tilde{\Delta}_l$, also their transversality in a continuous variation. If not, let need some more technical tools to show such properties in case, where we mention it in Section \S\ref{proof}.
Observe that in any case, Remark \ref{s-u NUH}, will guarantee us, the existence of well-defined local stable-unstable manifolds along any  orbits $\y\in O_{T}^{f}(x).$
\end{remark}

\textbf{\textit{Topological-Transitivity.}} An endomorphism $f: M\rightarrow M$ is said to be \textit{"Top-Transitive"} if for any $U,V$ open subsets of $M$, there exists some $n>0$ such that $f^{n}(U)\cap V \neq \emptyset$. 

\textbf{\textit{Weak Hyperbolic Transitivity.}} We say $x$ has a weak hyperbolic transitivity, if $O_T(x)$ is a weak hyperbolic subset of $M$. 

In above definition observe that, either the density of the forward orbit of $x$ or backward orbit of some $\x\in \pi^{-1}(x)$ implies the transitivity.

We may use the following definition of J. Milnor in examples section.

\begin{remark}[Measure-transitivity]\label{measure transitivity}
	Let $(M,f,\mu)$ being a measure dynamics (see \S \ref{preliminaries}). Then there is an almost equivalent definition of ergodicity which is "Measure-transitivity". A dynamical system is said to be measure transitive, if for any $U,V\subset M,$ where $\mu(U)>0,\,\mu(V)>0$, then there exist some $n>0$ such that $f^n(U)\cap V\neq\emptyset.$
	
	Observe that although this definition implies the top transitivity, but the inverse is not necessarily correct. 
\end{remark}



Let us bring the following Dichotomy of transitivity from \cite{LPV}.

\begin{lemma}[\textbf{Dichotomy of Transitivity}]. \label{LPV}
If $f\in C^{0}(M,M)$ then only one of the following properties hold: either $\omega_{f}(x)\neq M$ for all $x\in M$ or the set 
$\{x\in M: \omega_{f}(x)=M\}$ is a residual subset of $M$. Moreover when $f$ is transitive, then $\Gamma:=\{x\in M: \alpha_f(x)=\omega_{f}(x)=M\}$ is a residual subset of $M.$	
\end{lemma}

\begin{remark}\label{remark:LPV}
In above lemma the $\omega_{f}(x)$ and $\alpha_f(x)$ are the $\omega-$limit and $\alpha-$limit sets of a point $x\in M$ with their usual definition. 
\end{remark}

\textbf{\textit{Pre-Transitivity.}} An endomorphism $f: M\rightarrow M$ is said to be \textit{"Pre-Transitive"} if there exists some $x\in M$ with $O_T^{-}{x}$ (the set of pre-images of $x$ under $f$) is dense in $M$.

\begin{remark}
	Observe that one can shows that pre-transitivity and transitivity are somehow equivalent on a complete metric space. Despite this fact we prefer to use it just for case of periodic points. Meaning that pre-transitivity for us is assumed to be the existence of a dense total orbit of a periodic point and once for other points, we mention it. 
\end{remark}


\textbf{\textit{Robust Pre-Transitivity.}} A map $f:M\to M$ is said to be $C^1-$robustly pre-transitive, if there exists a $C^{1}$ open neighborhood $\mathcal{U}(f)$ in space of all $C^2-$regular maps, such that any $g\in \mathcal{U}(f)$ is pre-transitive. 

In other words it is enough to say that if $p_f$ is the pre-transitive periodic point of $f$, then there exist some pre-transitive periodic point $p_g$ for $g$ in a $C^1$ continuation of $f$.

\textbf{\textit{Roboust Weak Hyperbolic Transitivity.}} By a $C^1-$ roboust weak hyperbolic transitivity, we may refer to the existence of an open ball (in the $C^1-$ topology), where the weak hyperbolic transitivity is preserved.

\textbf{\textit{Topological-Mixing.}} We say that an endomorphism $f: M\rightarrow M$ has the property of \textit{"Top-Mixing"} if for any $U,V$ open subsets of $M$, there exists some $N>0$ large enough that for any $n>N$, the $f^{n}(U)\cap V \neq \emptyset$. 






\begin{lemma}\label{pre-transitivity}
Let $f$ be an endomorphism with pre-transitive periodic points, then some power of $f$ is top-mixing.
\end{lemma}

\begin{proof}
Let $f:M\to M$ be a dynamical system with dense periodic points such that each one has pre-transitivity. Let $U,V$ two open subsets of $M$. Let $q$ be some periodic point such that $f^m(q)=q$. Then $f^{m}$ is top-mixing. It comes from the fact that by pre-transitivity of $q$, there exists some $q_i,n$ such that $q_i\in V$ and $f^n(q_i)=q$. Therefore taking $f^m:=g$, for $N\geq mn$ we have that $g^{n}(V)\cap U\neq \emptyset$ for all $n>N.$
\end{proof}

\textbf{\textit{Statistical Stability.}} We say that $f$ is "Statistically Stable", if for a small perturbation, it persists its measures with SRB property (see Definition 3.5 from \cite{MT}). In other words we say that $f$, in the world of regular $C^{2}$ conservative maps is stably ergodic if there exists a $C^2$ neighborhood $\mathcal{V}(f)$ of $f$ such that Lebesgue measure is ergodic for any $g \in \mathcal{V}(f)$.

From a stochastic point of view, it is a conjecture that if there exists a unique SRB measure, then the system is stochastically stable. In a future work let tread this approach to see if a robust weak hyperbolic transitivity may also implies the stochastic stability of the dynamical system.  

\section{Some known Results}

In this section we mention some important definitions and results which are useful for the proof of theorems. 

Following definitions and result are from \cite{MT}.

\subsection{Ergodic Homoclinic Classes}\label{4-0}
Let $p \in M$ be a hyperbolic periodic point with period $n$. 
We define the \textit{Ergodic Homoclinic Class} of $p$ on both inverse limit space and $M.$ 
Recall that $\bar{p}$ is the unique periodic point of $\tilde{f}$ such that $\pi(\bar{p})=p.$

\begin{definition}[\textbf{EHC and $\widetilde{EHC}$}] \label{ergodic homoclinic classes}
	The inverse limit Ergodic Homoclinic Class ($\widetilde{EHC}$) is defined as $\tilde{\Lambda}(\bar{p}):=\tilde{\Lambda}^{s}(\bar{p})\cap\tilde{\Lambda}^{u}(\bar{p})$ where $\bar{p}= (\cdots, p, f(p), \cdots f^{n-1}(p), p, \cdots).$ 
	
	\begin{equation*}
	\tilde{\Lambda}^s(\bar{p}):=\{\x \in\tilde{\RR}\big{|} \exists n \geq 0 ,   W_{loc}^{s}(\tilde{f}^n(\x)) \pitchfork W^{u}(\mathcal{O}(\bar{p}))\neq \emptyset \}
	\end{equation*}
	and 
	\begin{equation*}
	\tilde{\Lambda}^u(\bar{p}):=\{ \x \in \tilde{\RR}\big{|} \exists \tilde{y} \in \tilde{\mathcal{R}}, \pi(\tilde{y}) = \pi(\x), \exists n \geq 0, \, \,  f^n(W^{u}(\y)) \pitchfork W_{loc}^{s}(\mathcal{O}(\bar{p}))\neq \emptyset\}.
	\end{equation*}
	Here $\tilde{\mathcal{R}}$ denotes the regular NUH-points in $M^f$ due to Remark \ref{regular NUH set}. 
	Observe that $\pi^{-1}(\pi (\tilde{\Lambda}^{*}(\bar{p}))) = \tilde{\Lambda}^{*}(\bar{p}) $ for $* \in \{s, u\}.$
	Once necessary we can define $\Lambda^s(p) := \pi(\tilde{\Lambda}^s(\bar{p}))$, $\Lambda^u(p) := \pi(\tilde{\Lambda}^u(\bar{p})) $ and $ \Lambda(p):= \pi(\tilde{\Lambda}(\bar{p}))$.
	\end{definition}
	
	Take $\x \in \tilde{\Delta}_l$ a recurrent point in the support of $\tilde{\mu}$ restricted to the Pesin block $\tilde{\Delta}_l$. Using Closing Lemma in \cite{MT}, we find a hyperbolic periodic point $\bar{p}$ therefore we have the following two crucial lemmas about the ergodic homoclinic class of $\bar{p}$.

	\begin{lemma}[\cite{MT}] \label{4-5}
		Let $p$ be a periodic point obtained as above, then $\tilde{\mu}(\tilde{\Lambda}(\bar p))>0.$	
	\end{lemma}
	
	\begin{lemma}[\cite{MT}]\label{4-6}
		$\tilde{\Lambda}(\bar p)$ is $\tilde{f}-$invariant.
	\end{lemma}
	
The following theorem is the main result of \cite{MT}.
	
	\begin{theorem} \label{pre-main}
		Let $f: M \rightarrow M$ be a $C^2-$endomorphism of a closed $n-$dimensional manifold. Then for any $ 0 < k < n$
		$$ \sharp \{ \text{Ergodic hyperbolic SRB measures of index k} \} \leq  \mathcal{E}_k.$$ 
	\end{theorem}
	
	To give a concrete bound for the number of SRB measures, it was defined the {\it k-skeleton} inside the set of hyperbolic periodic points of a fixed stable index. Let us bring a more precise version of this definition which will be used in the proof of the Theorem \ref{2} (find the first diffeomorphism and endomorphism versions correspondingly in \cite{VY}, \cite{MT}).
	
	\begin{definition}[\textbf{k-Eskeleton}] \label{def:skeleton}
		A $k-$skeleton $( 0 < k < n=\dim(M))$ of $f$ is a subset $P=\{p_j\}_{j\in\mathcal{I}}$ of hyperbolic periodic points $\{p_i\}_{i \in \mathcal{I}}$ of stable index $k$ such that:
		\begin{itemize}
			\item For any other $x\in M$ in special any  hyperbolic periodic point $p \in M \setminus P$, there is a \textit{"unique"} $j \in \mathcal{I}$ such that either $[p, p_j] \neq \emptyset$ or $[p_j, p] \neq \emptyset.$
			\item For every $i \neq j, [p_i, p_j] = \emptyset.$
			\end{itemize}
			\end{definition}
			
Denoting by $\mathcal{E}_k,$ the maximal cardinality of $k-$skeletons inside $Per_k(f)$ (hyperbolic periodic points of stable index $k$). Also for any two hyperbolic periodic points $p$ and $q$ we say that they are in a homoclinic relation or $[p, q] \neq \emptyset$ iff $ W^u(\bar{p}) \pitchfork W_{loc}^s(\mathcal{O}(q)) \neq \emptyset.$  If $z \in W^u(\bar{p}) \pitchfork W_{loc}^s(\mathcal{O}(q))$ then $T_z W^u(\bar{p}) \oplus T_z(W_{loc}^s(\mathcal{O}(q))) = T_z(M).$

\begin{lemma}
	Let $q$ be any hyperbolic periodic point of $M$ which does not belong to a $k-$eskeleton. Then $q$ becomes in a homoclinic relation with just one element of a $k-$eskeleton.
\end{lemma}

\begin{proof}
	The above lemma is a corollary of the Main Theorem of \cite{MT} and $\lambda-$ lemma. Let $K=\{p_1,...p_n,...\}$ any $k-$eskeleton. Suppose by contradiction that there exists more than one element of $K$ such that $q$ becomes in a homoclinic relation with them. Without loss of generality let us assume that there exists $p_1,p_2\in K$ such that $[p_1,q] \neq \emptyset$ and $[p_2, q] \neq \emptyset$. By definition this means that $ W^u(\bar{p_1}) \pitchfork W_{loc}^s(\mathcal{O}(q)) \neq \emptyset$ and $ W^u(\bar{p_2}) \pitchfork W_{loc}^s(\mathcal{O}(q)) \neq \emptyset$. Due to main results of \cite{MT}, $O(p_1),O(p_2)$ are the support of some disjoint ergodic SRB measures which we denote them by $\mu_{p_1},\mu_{p_2}$. Choosing adequate Pesin blocks and applying $\lambda-$ lemma to some unstable discs $D^u_1, D^u_2$ as pieces of $W^u(\bar{p_1})$ and $W^u(\bar{p_2})$ respectively, that intersect transversally with $W_{loc}^s(\mathcal{O}(q))$, we find some large positive $m,n$ such that $f^{m}(D^u_{1})$ and $W^u(D^u_{2})$ becomes near enough to $W^{u}(\bar q)$. Applying ergodic criteria technique, used in the proof of the main theorem of \cite{MT}, one can show that the basin of measures $\mu_{p_1}$ and $\mu_{p_2}$, intersect along the stable holonomies and so becomes non-empty. This is in a contradiction with the fact that the elements of an eskeleton carry disjoint SRB measures. In other words it says that $q$ should become in a homoclinic relation with just one element of the $k-$ eskeleton. 
\end{proof}

\section{Proof of the Main Results}\label{proof}

In this section we divide the proves into some lemmas to give a better understanding of them.

The main implements for the proof of Theorem 1 are the transitivity and hyperbolicity, which is not of uniform type obviously. As mentioned in the preliminary section, recall that $\tilde{\RR}$ is assumed to be the Pesin regular full $\tilde{\mu}$-measure subset of non-uniformly hyperbolic points achieved due to the assumption of the existence of an ergodic hyperbolic measure (In case of $\tilde{\mu}-$ positive, it is possible to normalize the measure to a full probability one).

In first step by following lemma let us show that a weak hyperbolic transitivity of $f$ on $M$ is equivalent to a weak hyperbolic transitivity on $M^{f}$.

\begin{lemma}\label{transitivity}
Let $(f,\mu)$ be a hyperbolic measure dynamics with weak hyperbolic transitivity of some point $x\in M$. Then $O_{T}^{f}(x)$ is dense on $M^{f}$.		
\end{lemma}

\begin{proof}

To show the density of $O_T^{f}(x)$ (see the definition of $\Lambda^{f}$ in preliminary part), let $\y$ be any point which belongs to $M^{f}$, then $\pi(\y):=y$ is a point in $M$.

For any small $\epsilon-$ball $B_{\epsilon}(y)\subset M$, by compactness of $M$ and continuity of $f$, being a local diffeomorphism, we always can choose $m>0$ large enough such that for a large $N$ number of iterates $d(f^{m+i}(x), f^{i}(y)<\epsilon, |i|\leq N$ and then by choosing suitable pre-images of $f^{m}(x),$ which are close enough to $(y_{-n})$ (pre-images along the $\y$ branch) we find some $\x\in O_{T}^{f}(x)$ that $\tilde{d}(\x,\y)<\epsilon.$ As $\y\in M^{f}$ was arbitrary and all such $\x$ would belong to the limit set of $O_T{x}$, then we conclude that $O_T^{f}(x)$ is dense on the inverse limit space $M^{f}$ and in special is dense in $\tilde{\RR}$.
  
\end{proof}

What seems important here, is showing that for $O_{T}^{f}(x)$ dense in $\tilde{\RR}$, its elements approximating Pesin blocks, their local stable-unstable manifolds will transverse the local stable-unstable manifold of some element of the Pesin block. As mentioned in Remark \ref{important}, it may happen that $O_{T}^{f}(x)$ or some of its elements belong to the closure of Pesin blocks, which makes the case easy and we have to do nothing with it. But let use the assumption of Lemma \ref{transitivity} and suppose that they do not belong to the closure of any Pesin Blocks, yet it happens that for some point close enough to the Pesin blocks we get the transversality condition.

\begin{lemma}\label{transversality}
Let $(f,\mu)$ be a hyperbolic measure dynamics with weak hyperbolic transitivity of some point $x\in M$. Then there exist some $\y\in O_{T}^{f}(x)$ and some $\r\in\tilde{\Delta}_{l}$ (arbitrary Pesin block with $l>0$), such that happens 
$$W^{s}_{loc}(\y)\pitchfork W^{u}_{loc}(\r)\neq \emptyset,\,\,\,W^{u}_{loc}(\y)\pitchfork W^{s}_{loc}(\r)\neq \emptyset.$$
\end{lemma}

\begin{proof}
The proof of this lemma is somehow technical and we refer to \cite{MT2} and \cite{8}. We use the same technique which was used by A.Katok in his paper \cite{8} to show that the local stable-unstable manifolds of a periodic point derived from Katok Closing lemma, are the admissible manifolds near the base point which this phenomena happens. More precisely we use the similar technique, one we used in \cite{MT2}, for the proof of part III of the Main Lemma 1.1.

Omitting the basic definitions, we refer to \cite{MT}, \cite{MT2} for definition of Lyapunov charts for endomorphisms also the definition of $(s,h,.)$ or $(u,h,.)$-stable and unstable admissible manifolds respectively. Let $l>0$ given and $\tilde{\Delta}_{l}$ any Pesin block of $\QQ$. Let consider $\r\in \tilde{\Delta}_{l}$ which belongs to the closure of the block. Inside a $\r-$ Lyapunov chart $R(\r,h)$, where $h>0$ is a small value, we may define the $(u,h,\r)$ and $(s,h,\r)$ admissible manifolds. 

By definition, admissible manifolds are graph of some $C^1-$local maps in $\r-$ Lyapunov charts, defined on some small neighborhoods of origin. 


We remind that by Lemma \ref{transitivity}, the $O_{T}^{f}(x)$ is dense in $M^{f}$ and specially in $\QQ$.
Now considering the fact that $M^{f}$ is a compact metric space, for $\r\in \tilde{\Delta}_{l}$, there exist some sequence of elements in $O_{T}^{f}(x)$, which we may denote it by $\{\y^{j}\}_{j\in \NN},$ that converges to $\r.$ We choose $\y\in \{\y^{j}\}_{j\in \NN}$ being the element which is close enough to $\r$ where belongs to the Lyapunov chart $R(\r,h/10).$ Now using this known fact that local stable- unstable manifolds are the graph of some $C^1-$ functions, we derive that the local stable- unstable manifolds of $\y\in R(\r,h/10)$, are $(u, h/10, \r)$ and $(s, h/10, \r)$ admissible manifolds near the point $r\in M.$ Using Proposition 4.2 of \cite{MT} or Proposition 4.8 of \cite{MT2}, we conclude that desired transversalities between local stable- unstable manifolds occur.
\end{proof}

In the next step, let us use the definition of $\widetilde{EHC}$ and the independence property of local stable manifolds in respect to different orbit branches of a point, in order to show that if $x$ is a weak hyperbolic transitive point of $f$, then not only there exist some $\x\in \pi^{-1}(x)\subset O_{T}^{f}(x)$, that belongs to every ergodic homoclinic class $\tilde{\Lambda}(\p)$ of type \ref{4-5}, but also, all $\x\in \pi^{-1}(x)$ belong to all such $\widetilde{EHC}$. Obviously this last helps us establishing the uniqueness of hyperbolic SRB measures, due to  \cite{MT}. 

\begin{lemma}\label{total transitivity in EHC}
Let $x$ be a transitive weak hyperbolic point of $f$, then all $\x\in\pi^{-1}(x)$ belongs to all ergodic homoclinic classes of some hyperbolic periodic points.
\end{lemma}

\begin{proof}

Let $\tilde{\Delta}_{l}$ be some suitable Pesin block, where due to Lemma 5.1 of \cite{MT} plays the role of an ergodic homoclinic class of some hyperbolic periodic point, $\tilde{\Lambda}(\p)$ on $M^{f}$. By Lemma \ref{4-5} we know that $\tilde{\Lambda}(\p)$ has a positive $\tilde{\mu}-$ measure. In Lemma \ref{transitivity} we have shown that weak hyperbolic transitivity of some $x\in M$ implies the density of $O_T^{f}(x)$ on $\tilde{\QQ}$. 

To show that there exist some $\x\in O_T^{f}(x)\cap\tilde{\Lambda}(\p)$. Observe that $\tilde{\mu}(\tilde{\Lambda}(\p))>0$ and $O_{T}^{f}(x)$ is dense in $\tilde{\QQ}=\bigcup_{l>1}\overline{\Delta_{l}}$. Therefore there exist some $m\in \ZZ$ that $\y=\tilde{f}^{m}(\x)\in \tilde{\Lambda}(\p)\cap \tilde{\Delta}_{l}$, where $\x\in \pi^{-1}(x)$. Due to Lemma \ref{4-6}, $\tilde{\Lambda}(\p)$ is $\tilde{f}-$ invariant, thereupon we can conclude that there exists some $\x\in \pi^{-1}(x)$ which belongs to $\tilde{\Lambda}(\p)$.

Now let us show that not only there exist some $\x\in \pi^{-1}(x)$ which  belongs to $\tilde{\Lambda}(\p)$, that also all $\x\in \pi^{-1}(x)$ belongs to $\tilde{\Lambda}(\p)$. To see that, use the definition of ergodic homoclinic classes on $M^{f}$. Suppose $\x_i$ is any $\pi-$pre-image of $x$ which belongs to $\pi^{-1}(x)$. We need to show that $\x_i$ belongs to both $\tilde{\Lambda}^s(\bar{p})$ and $\tilde{\Lambda}^u(\bar{p})$. As $\x\in \tilde{\Lambda}(\p)$ then $\x\in\tilde{\Lambda}^s(\bar{p})$, hence there exists some $m>0$ such that 
$$W_{loc}^{s}(\tilde{f}^n(\x)) \pitchfork W^{u}(\mathcal{O}(\bar{p}))\neq \emptyset,$$

Observe that the local stable manifolds does not depend on $\x$, which means $W_{loc}^{s}(\tilde{f}^n(\x))=W_{loc}^{s}(\tilde{f}^n(\x_i))$ for any $\x_i\in \pi^{-1}(x)$, means we have 

$$W_{loc}^{s}(\tilde{f}^n(\x_i)) \pitchfork W^{u}(\mathcal{O}(\bar{p}))\neq \emptyset\Longrightarrow \x_i\in \tilde{\Lambda}^s(\bar{p}).$$

It remains to show that $\x_i$ belongs to $\tilde{\Lambda}^u(\bar{p})$. In fact this part is somehow trivial by definition. We know that $\x \in \tilde{\Lambda}^u(\bar{p})$, so by definition, it exists some $\y\in \pi^{-1}(x)$ that $\exists n \geq 0$ that $f^n(W^{u}(\y)) \pitchfork W_{loc}^{s}(\mathcal{O}(\bar{p}))\neq \emptyset$. By the fact that $\pi(\x)=\pi(\x_i)$ thereupon $\x_i\in\tilde{\Lambda}^u(\bar{p})$. 
\end{proof}

\vspace{0.1cm}
\begin{proof}[\textbf{Proof of Theorem 1}]
According to Theorem 5.10 of \cite{MT} there is a close relationship between hyperbolic measures with SRB property and ergodic homoclinic classes of some periodic points. It says that to any such measure $\mu$ and its inverse limit correspondence $\tilde{\mu},$ we can correspond some hyperbolic periodic point $p=\pi(\p)$ that $\tilde{\mu}|_{\tilde{\Lambda}(\p)}$ is ergodic. To obtain the uniqueness of corresponded SRB measures, it is enough to show that considering a weak hyperbolic transitivity, all such $\tilde{\Lambda}(\p)$ coincides.

Due to above lemmas, $O_{T}^{f}(x)$ is dense in $\tilde{\RR}$. Also for any hyperbolic periodic point $p$ corresponded to a hyperbolic measure with some SRB property, $O_{T}^{f}(x)$ becomes dense in $\tilde{\Lambda}(\p)$. 
Suppose $p,q$ any arbitrary hyperbolic periodic points supporting such hyperbolic ergodic SRB measures. We claim that under the weak hyperbolic transitivity of $O_{T}^{f}(x)$, their ergodic homoclinic classes, coincide.

We shall prove that $\tilde{\Lambda}(\p)=\tilde{\Lambda}(\q)$. Using Lemma \ref{total transitivity in EHC}, $\pi^{-1}(x)$ is dense in $\tilde{\Lambda}(\p)$ also in $\tilde{\Lambda}(\q)$. Let $\tilde{z}$ any point of $\tilde{\Lambda}(\p)$, we need to show that $\tilde{z}\in\tilde{\Lambda}(\q)$. By density property, take any sequence $\{\x^{i}\}$ of $\pi^{-1}(x)\cap\tilde{\Lambda}(\q)$ such that $\tilde{d}(\x^{i},\tilde{z})\xrightarrow{i\to \infty} 0$, then we claim that $\tilde{z} \in\tilde{\Lambda}(\q)$. This comes from the fact that inside the Pesin blocks the local stable-unstable manifolds vary continuously and this implies that once all $\x^{i}\in \tilde{\Lambda}(\q)$ then $\z\in \tilde{\Lambda}(\q)$ too. By a similar approach one can show that
$\tilde{\Lambda}(\q)\subset \tilde{\Lambda}(\p)$ and this finishes the claim as well as the proof of the theorem and uniqueness of the ergodic hyperbolic measures with SRB property. 
\end{proof}

Now let us give the proof of other theorems. 

\vspace{0.1cm}
\begin{proof}[\textbf{Proof of Theorem 2}]
	Once having weak hyperbolic pre-transitivity, we show that nominated pre-transitive periodic point, let call it $p$, becomes in a relation with all other periodic points in special the ones carrying the support of hyperbolic SRB measures. This comes from the fact that the stable set of $O(p)$ spreads on its pre-periodics which in case is dense. For instance let $q$ any hyperbolic periodic point carrying the support of some hyperbolic SRB measure. Also let suppose that $\bar q$ is its unique inverse limit periodic object. Then one is able using the continuity of $f$, construct some $\tilde{p}\in O_{T}^{f}(p)$ which is close enough to $\q$ such that by Lemma \ref{transversality} succeed the transversality between the local stable manifold of $\tilde{p}$ with $W^{u}(\bar q)$. Thereupon $p$ and $q$ becomes in a homoclinic relation, meaning that the cardinality of the skeleton set, can be just one, which signify the uniqueness of SRB measure by Theorem \ref{pre-main}.
\end{proof}

\vspace{0.1cm}
\begin{proof}[\textbf{Proof of Theorem 3}] This theorem is a simple consequence of Theorem \ref{main1}. It drives that in a conservative setting the unique hyperbolic measure is Lebesgue measure, which indicates the ergodicity of the dynamical system.
\end{proof}

\vspace{0.1cm}
\begin{proof}[\textbf{Proof of Corollary 3.1}] The fact that a $C^{1}-$ robust weak hyperbolic transitivity implies the statistic stability is also, a corollary of Theorem \ref{main1}. Let $g\in\mathcal{V}$ be a $C^{1-}$ small perturbation of $f$ in a neighborhood of $f\in \mathcal{C}^{r}(M)$. Once $g$ preserves the weak hyperbolic transitivity, then due to Theorem \ref{main1}, it has the unique SRB measure property, which implies the statistic stability of $f$. Whenever being in a conservative setting, it bears the stable ergodicity.
	
\end{proof} 

\section{Examples} 
In this section we try to mention class of examples which the main theorem may apply. We see the uniqueness of SRB measures or exact number of SRB measures for some known examples using above theorems and give a deliberation on their statistic stability or ergodic stability in case the SRB measure is the Lebesgue measure. 

\subsection{Transitive Anosov maps}

According to \cite{14}, \cite{PM} on one hand endomorphisms are not structurally stable and from other hand there exist further results \cite{Q}, \cite{Ikeda}, \cite{BC} where demonstrate the structural stability of endomorphisms on their inverse limit space. In special P. Berger and A. Kocsard, in their recent work show that Axiom A endomorphisms with a strong transversality condition are inverse limit stable and vice versa.

One of the first examples we can mention here are Anosov endomorphisms on the torus. It is well-known that Anosov maps of the torus are weak hyperbolic transitive  (\cite{AH} see Theorem 8.3.4). So on by the main result of this paper, transitive Anosov maps have a unique hyperbolic SRB measure, moreover in Lemma \ref{transitivity}, we show that a weak hyperbolic transitivity is equivalent to a weak hyperbolic transitivity on the inverse limit space which using the structural stability of such endomorphisms on the inverse limit space, it arises that such Anosov are robustly weak hyperbolic and transitive which due to Theorem \ref{main3}, they become examples with statistical stability. In case like Spacial Anosov endomorphisms of torus which are conservative maps, the theorem implies their ergodic stability.

It is notable that in a recent work, M. Anderson shows that the homotopy class of hyperbolic and conservative endomorphisms of the torus, consist entirely of transitive maps.

\subsection{The Absence Of Hyperbolic-Transitivity for Kan Type Examples}\label{kan}
The Kan's example is a local diffeomorphism $F$ defined on the cylinder $\mathbb{S}^1 \times [0,1]$ as a skew product:
$$
F(z, t):= (z^d, f_z(t)), 
$$ where $z \in \mathbb{S}^1$ is a complex number of norm one and $z^d$ is the expanding  covering of the circle of degree $d > 2$. For each $z \in \mathbb{S}^1$ the function $f_z: [0, 1] \rightarrow [0,1]$ is a diffeomorphism fixing the boundary of $[0, 1].$ Take two fixed points of $z^d$ called $p, q.$ We require that $f_p$ and $f_q$ have exactly two fixed points each, a source at $t=1$ (respectively $t=0$) and a sink at $t=0$ (respectively $t=1$). Furthermore, $|f_z^{'} (t)| < d$ and 
$$
\int \log f_z^{'}(0) dz < 0 \quad \text{and} \quad  \int \log f_z^{'}(1) dz < 0.
$$ 
Under these conditions $F$ has two intermingled SRB measures which are normalized Lebesgue measure of each boundary circle. Under some more conditions $F$ is not only transitive that also robustly transitive (see \cite{BDV}).
We can consider two such examples and glue them to find a local diffeomorphism of $\mathbb{T}^2$ admitting two SRB measures also with topological transitivity but not robustly anymore. In paper \cite{UV}
they show that a 3 dimensional torus with two intermingled basins is not robust anymore (in a strong partially hyperbolic setting).

\begin{figure}
	\centering
	\includegraphics[width=0.45\linewidth]{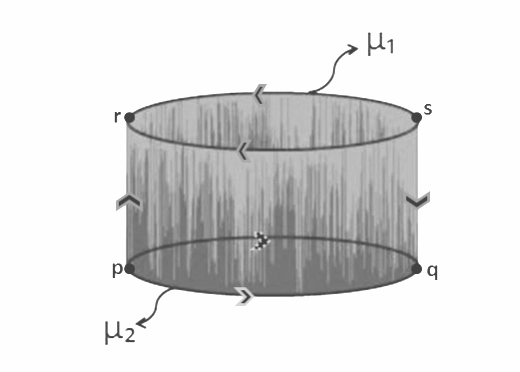}
	\caption{Two intermingled basins of the Kan example.}
	\label{fig:Kan}
\end{figure}

Let denote the two SRB measures with intermingled basins in Kan example with $\mu_1,\mu_2$. 
Due to our main result, the Kan type examples either bearing the top-transitivity but the forward dense orbit can not be of any weak hyperbolic type. John Milnor, in some of his unpublished dynamics notes, shows that I. Kan example on cylinder, either having top-transitivity is not a measure-transitive example. As measure transitivity is equivalent to ergodicity, it is not difficult to see that the 2-dim Lebesgue measure is a non-ergodic hyperbolic measure and the two invariant intermingled basin of attractions $B(\mu_1),B(\mu_2)$ have positive Lebesgue measures. Not being a measure transitive here plus the hyperbolicity means that you will not find any hyperbolic point which belongs to both of the basins of attractions. 
Whereas it is a well-known fact that Kan examples does not have a mixing property, we believe that in such examples in fact by a small perturbation, one can obtain a weak hyperbolic transitivity which leads to a unique SRB property and significantly the reason of non-robustness property. In this way one may obtain a weak hyperbolic transitive example which does not have a mixing property.
 



\subsection{Derived from Anosov Maps}
The third class of weak hyperbolic robustly transitive examples are an open set of Derived-from Anosov examples which appears in \cite{18}. In Theorem B of this paper N. Sumi shows that the class of all  $C^{1}-$regular maps of the torus, contain a nonempty $C^{1}$ open set $\mathcal{U}$, such that every regular map belonging to $\mathcal{U}$ is a DA-map with topologically mixing property.

Such class of examples possessing both of the properties, topologically mixing and weak hyperbolic transitivity, by Theorem\ref{main1} not only have unique SRB measure, that also have $C^1-$ stable ergodicity. 

\begin{remark}
In both cases of Anosov endomorphisms and Derived from Anosov case, one can show the existence of hyperbolic pre-transitivity (a hyperbolic periodic point with dense set of pre-images) and use Theorem \ref{2}, to obtain the uniqueness property of the SRB measures.
\end{remark}

\end{document}